\documentclass[11pt]{article}
\usepackage[T1]{fontenc}
\usepackage[utf8]{inputenc}
\usepackage{lmodern}
\usepackage[english]{babel}
\usepackage{microtype}
\usepackage{amsmath,amssymb,amsfonts,amsthm,mathtools}
\usepackage[shortlabels]{enumitem}
\usepackage{geometry}
\usepackage[colorlinks=true,linkcolor=blue,citecolor=blue,urlcolor=blue]{hyperref}
\newtheorem{theorem}{Theorem}[section]
\newtheorem{lemma}[theorem]{Lemma}
\newtheorem{claim}[theorem]{Claim}
\newtheorem{proposition}[theorem]{Proposition}
\newtheorem{corollary}[theorem]{Corollary}

\newtheorem{conjecture}[theorem]{Conjecture}
\newcommand{\eps}{\varepsilon}

\newcommand{\ceil}[1]{\lceil #1\rceil}

\newcommand{\1}{\mathbf 1}

\usepackage{xcolor}

\usepackage{extarrows}
\title{Asymptotic Uniformity of Permanents of Random Matrices over Finite Fields of Odd Characteristic}

\author{Shuang Sun\thanks{Email: \texttt{chocolatesun@sjtu.edu.cn}},\quad Yuyao Yang\thanks{Email: \texttt{alaia\_y@sjtu.edu.cn}},\quad and Jiasheng Zeng\thanks{Email: \texttt{jasonzeng@mail.ustc.edu.cn}}}
\date{\today}

\begin{document}
\maketitle
\begin{abstract}
Let $q$ be an odd prime power, and let $A_n=(a_{ij})\in\mathbb F_q^{n\times n}$ be a random matrix whose entries are independent and uniformly distributed on $\mathbb F_q$. The permanent of $A_n$ is defined by $\operatorname{per}(A_n)=\sum_{\sigma\in S_n}\prod_{i=1}^n a_{i,\sigma(i)}$, where $S_n$ denotes the symmetric group on $[n]$. Ghasemi, Gross, and Kopparty conjectured the zero-mass asymptotic $\Pr[\operatorname{per}(A_n)=0]=1/q+o(1)$ for every fixed odd prime power $q$, and Hunter, Kwan, and Sauermann subsequently stated its equivalent full-distribution formulation: for every fixed $q$ and every $x\in\mathbb F_q$,
\[
\lim_{n\to\infty}\Pr[\operatorname{per}(A_n)=x]=\frac1q.
\]
In this paper, we prove this conjecture. More precisely, we prove that there is an absolute constant $C>0$ such that \[\frac12\sum_{x\in\mathbb F_q}\left|\Pr[\operatorname{per}(A_n)=x]-\frac1q\right|\le C\frac{\log n}{n}\] for every odd prime power $q$ and every $n\ge 7$. The estimate is uniform in $q$, so the conclusion remains valid for every sequence $q=q(n)$ of odd prime powers.
\end{abstract}

\section{Introduction}\label{introduction}
 For a matrix $A=(a_{ij})\in\mathbb F_q^{n\times n}$, the determinant and the permanent of $A$ are defined by
\begin{equation}
\det(A)=\sum_{\sigma\in S_n}\operatorname{sgn}(\sigma)\prod_{i=1}^n a_{i,\sigma(i)}
\qquad\text{and}\qquad
\operatorname{per}(A)=\sum_{\sigma\in S_n}\prod_{i=1}^n a_{i,\sigma(i)}.
\label{definitions}
\end{equation}
Standard references for the permanent and its relation to the determinant include the monograph of Minc and the classical work of Marcus and Minc \cite{Minc1978,MarcusMinc1961}. The similarity in \eqref{definitions} conceals a basic structural difference. The determinant is controlled by linear algebra, while the permanent generally is not.

This distinction is already visible for random matrices over finite fields. A uniform matrix in $\mathbb F_q^{n\times n}$ is singular with probability $1-\prod_{j=1}^n(1-q^{-j})$. Conditional on invertibility, its determinant is uniform on $\mathbb F_q^\times$. Thus the limiting determinant distribution has mass $\alpha_q=1-\prod_{j=1}^{\infty}(1-q^{-j})$ at zero and mass $(1-\alpha_q)/(q-1)$ at every nonzero value. Fulman's survey gives a broad account of random matrix theory over finite fields \cite{Fulman2002}. For matrices whose entries are independent but not uniform, Kahn and Koml\'os established far-reaching universality for singularity probabilities, and Luh, Meehan, and Nguyen developed modern inverse and equidistribution methods for rank and spectral statistics \cite{KahnKomlos2001,LuhMeehanNguyen2021}. These determinant and rank results provide an important benchmark, but they rely on linear structure that is unavailable for the permanent. Even in the uniform model, the limiting determinant distribution is not uniform for fixed $q$ because $\alpha_q>1/q$.

In characteristic $2$, the signs in the determinant disappear and $\operatorname{per}(A)=\det(A)$. If $q$ is even and $p_{n,q}=\prod_{j=1}^n(1-q^{-j})$, then $\Pr[\operatorname{per}(A)=0]=1-p_{n,q}$ and every nonzero value has probability $p_{n,q}/(q-1)$. Consequently, the permanent cannot become uniform on a fixed field of characteristic $2$. Odd characteristic is therefore necessary for the question studied here.

The permanent is central in enumerative combinatorics and computational complexity. Valiant proved that its exact computation is $\#\mathrm P$-complete \cite{Valiant1979}, while Jerrum, Sinclair, and Vigoda obtained a fully polynomial randomized approximation scheme for matrices with nonnegative entries \cite{JerrumSinclairVigoda2004}. Over large finite fields, the permanent polynomial also became a principal example in random self-reducibility and polynomial correction. This line includes work of Lipton, Gemmell and Sudan, and Gemmell, Lipton, Rubinfeld, Sudan, and Wigderson \cite{Lipton1991,GemmellSudan1992,GemmellLiptonRubinfeldSudanWigderson1991}. Feige and Lund established strong average-case hardness results for random matrices over suitably large prime fields \cite{FeigeLund1996}. 
There is a parallel probabilistic literature on permanents with discrete random entries. Tao and Vu proved a strong nonvanishing result for random Bernoulli matrices \cite{TaoVu2009}. Kwan and Sauermann treated random symmetric matrices \cite{KwanSauermann2022}. More recently, Hunter, Kwan, and Sauermann proved exponential anticoncentration under a general atom bound \cite{HunterKwanSauermann2025}. Those works take values in characteristic zero and address nonvanishing or anticoncentration. The finite-field problem asks instead whether all field values eventually receive nearly equal mass.

Another direct predecessor comes from P\'olya's permanent problem, which asks when signs can transform a permanent into a determinant. The subject includes the structural theory of Pfaffian orientations developed by Robertson, Seymour, and Thomas \cite{RobertsonSeymourThomas1999}. Over finite fields, Dolinar, Guterman, Kuzma, and Orel studied the corresponding determinant-permanent relation \cite{DolinarGutermanKuzmaOrel2011}. Budrevich and Guterman proved that, for odd $q$ and $n\geq3$, the permanent vanishes on fewer matrices than the determinant, and Bassalygo later gave a shorter proof \cite{BudrevichGuterman2012,Bassalygo2013}. Budrevich subsequently obtained further enumerative estimates for matrices with nonzero permanent \cite{Budrevich2018}. These results establish a strict comparison, although their gap does not identify the limiting permanent distribution.

Finite-field permanents have also been studied in other asymptotic regimes. Lyapkov and Sevast'yanov \cite{LyapkovSevastyanov1996} proved a limit theorem for the generalized permanent of a random $n\times m$ matrix over $\operatorname{GF}(p)$ with independent rows, in the regime $n\to\infty$ with $m$ fixed. Its limiting law is a mixture of a point mass at zero and the uniform law on $\operatorname{GF}(p)$. Vinh \cite{Vinh2012} obtained value-set and distributional estimates in fixed-dimension, growing-field regimes, including matrices with entries restricted to large subsets. These results concern, respectively, a fixed-width rectangular limit and fixed-dimension growing-field or restricted-entry regimes. 
For square matrices over a fixed finite field, the earliest explicit question known to us is Kopparty's Problem~9 from the 2017 Oberwolfach workshop on combinatorics. It asks for $\Pr[\operatorname{per}(A_n)=0]$ when $A_n$ is a random $n\times n$ matrix over $\mathbb F_3$ \cite{KoppartyProblem2017}. Scheinerman later gave extensive computational evidence for asymptotic uniformity in this case \cite{Scheinerman2024}.

Ghasemi, Gross, and Kopparty subsequently formulated the general fixed-odd-field conjecture in the zero-mass form $\Pr[\operatorname{per}(A_n)=0]=1/q+o(1)$ \cite{GhasemiGrossKopparty2025}. Their main theorem treats random $n\times k$ matrices with $k\leq0.1\sqrt n$ and therefore does not reach the square case $k=n$. Hunter, Kwan, and Sauermann later stated the equivalent full-distribution formulation \cite[Conjecture~1.2]{HunterKwanSauermann2026}. Indeed, multiplying one fixed row by an element of $\mathbb F_q^\times$ is a measure-preserving bijection and multiplies the permanent by the same element. Hence all nonzero values have the same probability, and convergence of the zero mass to $1/q$ is equivalent to convergence of the entire distribution to the uniform law. 

\begin{conjecture}[Finite-field permanent uniformity, see \cite{GhasemiGrossKopparty2025,HunterKwanSauermann2026}]\label{uniformityconjecture}
Fix a finite field $\mathbb F_q$ of odd characteristic. Then, for a uniformly random $n\times n$ matrix $A_n\in\mathbb F_q^{n\times n}$ and every $x\in\mathbb F_q$, one has
\begin{equation}
\lim_{n\to\infty}\Pr[\operatorname{per}(A_n)=x]=\frac1q.
\label{uniformitystatement}
\end{equation}
\end{conjecture}

Hunter, Kwan, and Sauermann proved $\Pr[\operatorname{per}(A_n)=0]\geq1/q$. They also proved that $\Pr[\operatorname{per}(A_n)=0]\leq1/q+C/q^3$ for some absolute constant $C>0$ and $n\geq3$, and $\limsup_{n\to\infty}\Pr[\operatorname{per}(A_n)=0]\leq\alpha_q-1/(50q^2)$ \cite[Theorem~1.3]{HunterKwanSauermann2026}. The first estimate is strong when the field is large, while neither estimate tends to $1/q$ with $n$ when the field is fixed.

For probability measures $\mu$ and $\nu$ on $\mathbb F_q$, write $d_{\mathrm{TV}}(\mu,\nu)=\frac12\sum_{x\in\mathbb F_q}|\mu(x)-\nu(x)|$. We write $\mathsf U_q$ for the uniform probability measure on $\mathbb F_q$ and $\mathcal L(X)$ for the law of a random variable $X$. Our main result resolves the fixed-field square-matrix conjecture, with an estimate uniform in the field size. 

\begin{samepage}
\begin{theorem}\label{maintheorem}
Let $q$ be an odd prime power, let $n\geq1$, and let $A_n\in\mathbb F_q^{n\times n}$ be uniformly random. For every integer $t$ with $0\leq t\leq n-1$, one has
\begin{equation}
d_{\mathrm{TV}}\!\left(\mathcal L(\operatorname{per}(A_n)),\mathsf U_q\right)
\leq
\frac{(q-1)(2q-1)}{q^3}\frac{t}{n}
+
\frac{1}{q^2}\left(1-\left(1-\frac{1}{q}\right)^3\right)^t.
\label{parameterbound}
\end{equation}
\end{theorem}

\begin{corollary}\label{uniformitycorollary}
Conjecture~\ref{uniformityconjecture} holds. More precisely, for every odd prime power $q$, every $n\geq7$, and uniformly random $A_n\in\mathbb F_q^{n\times n}$, one has
\begin{equation}
d_{\mathrm{TV}}\!\left(\mathcal L(\operatorname{per}(A_n)),\mathsf U_q\right)
\leq
\frac{1}{n}\left[
\frac{10}{27}\left(\frac{\log n}{\log(27/19)}+1\right)+\frac{1}{9}
\right].
\label{uniformbound}
\end{equation}
Consequently, the total-variation distance is $O((\log n)/n)$ uniformly over all odd prime powers. The same conclusion holds for every sequence of odd prime powers $q=q(n)$.
\end{corollary}
\end{samepage}

The estimates of Hunter, Kwan, and Sauermann and Theorem~\ref{maintheorem} are complementary. Their proof gives error at most $11/q^3$ in the zero probability for $n\geq3$, while \eqref{uniformbound} gives decay in $n$ uniformly in $q$. The exact mixture identity proved in Section~\ref{iterationsection} shows that this zero-probability error is also the total-variation distance. Taking the better of the two estimates gives useful control throughout the two-parameter range. When $q$ is large relative to $n$, the bound of Hunter, Kwan, and Sauermann is often stronger. When $q$ is fixed or changes without tending to infinity, \eqref{uniformbound} supplies the missing decay in the matrix dimension. The latter feature is essential for $q=3$, the case of Kopparty's 2017 problem, and it also permits arbitrary oscillation of $q=q(n)$ among odd prime powers.

The proof starts from a commutative algebra whose generators have square zero. A product of random linear forms in this algebra simultaneously records all maximal permanental minors. A nonzero intermediate product determines a quotient of the degree-one part. The coordinate classes in that quotient lie on a nonzero quadratic hypersurface. A sharp finite-field zero count forces many coordinate collisions. The identity $\operatorname{Ann}(x_i-x_j)=(x_i+x_j)R_n$ converts every off-diagonal collision into the vanishing of a lower-dimensional product. This gives a recurrence in the matrix dimension and the number of omitted rows. Its iteration for $t$ steps produces \eqref{parameterbound}, and a choice of $t$ of order $\log n$ gives Corollary~\ref{uniformitycorollary}.

\medskip
The rest of the paper is organized as follows. 
Section~\ref{algebraicsection} develops the algebraic encoding and the two zero estimates. Section~\ref{collisionsection} proves the collision identity and the dimension-reduction recurrence. Section~\ref{iterationsection} iterates the recurrence and completes the proof.

\section{Algebraic encoding and zero estimates}\label{algebraicsection}
Fix an odd prime power $q$. For each positive integer $n$, write $[n]=\{1,\ldots,n\}$ and define $R_n=\mathbb F_q[x_1,\ldots,x_n]/(x_1^2,\ldots,x_n^2)$. We use $x_i$ also for the image of the indeterminate $x_i$ in $R_n$, so $x_i^2=0$ for every $i\in[n]$. For each $S\subseteq[n]$, let $x_S=\prod_{j\in S}x_j$, with $x_\varnothing=1$. Every monomial in which some variable occurs more than once is zero in $R_n$, and hence every element of $R_n$ has a unique expression $\sum_{S\subseteq[n]}c_Sx_S$ with $c_S\in\mathbb F_q$. It follows that the elements $x_S$ with $S\subseteq[n]$ form an $\mathbb F_q$-basis of $R_n$. For $0\leq d\leq n$, let $(R_n)_d$ be the $\mathbb F_q$-span of the elements $x_S$ with $|S|=d$ and 
$(R_n)_d={0}$ for $d>n$. For $S,T\subseteq[n]$, one has $x_Sx_T=0$ when $S\cap T\neq\varnothing$ and $x_Sx_T=x_{S\cup T}$ when $S\cap T=\varnothing$. Consequently, when a product of elements of $(R_n)_1$ is expanded, the terms that use any variable more than once vanish, while the coefficients of the remaining monomials are permanents of the corresponding submatrices. The following lemma gives the precise identity used in the proof, with related formulations appearing in \cite{FeinsilverMcSorley2011}.

\begin{lemma}\label{encodinglemma}
Let $1\leq m\leq n$, let $M=(m_{ij})$ be an $m\times n$ matrix over $\mathbb F_q$, and let $L_i=\sum_{j=1}^n m_{ij}x_j$. If $M_S$ denotes the submatrix formed by the columns indexed by $S$, then
\begin{equation}
L_1\cdots L_m
=
\sum_{\substack{S\subseteq[n]\\|S|=m}}
\operatorname{per}(M_S)x_S.
\label{encoding}
\end{equation}
\end{lemma}

\begin{proof}
Expanding $L_1\cdots L_m$ requires choosing one summand from each factor $L_i=\sum_{j=1}^n m_{ij}x_j$. Such a choice is determined by a map $\phi$ from $[m]$ to $[n]$, where the summand chosen from $L_i$ is $m_{i,\phi(i)}x_{\phi(i)}$. The resulting term is $\left(\prod_{i=1}^m m_{i,\phi(i)}\right)\left(\prod_{i=1}^m x_{\phi(i)}\right)$. Suppose that $\phi$ is not injective. Then there are distinct $r,s\in[m]$ such that $\phi(r)=\phi(s)=j$ for some $j\in[n]$. The monomial $\prod_{i=1}^m x_{\phi(i)}$ then contains the factor $x_j^2$, so it is zero in $R_n$. Hence only injective maps can contribute nonzero terms.

Suppose that $\phi$ is injective and let $S=\{\phi(1),\ldots,\phi(m)\}$. Then $|S|=m$, and since the variables commute, $\prod_{i=1}^m x_{\phi(i)}=\prod_{j\in S}x_j=x_S$. Thus the coefficient of $x_S$ is obtained by summing the scalar coefficients associated with all injective maps whose image is $S$. Fix $S\subseteq[n]$ with $|S|=m$ and write $S=\{s_1<\cdots<s_m\}$. Every injective map from $[m]$ to $[n]$ with image $S$ has the form $\phi(i)=s_{\sigma(i)}$ for a unique permutation $\sigma\in S_m$. Since the columns of $M_S$ are the columns of $M$ indexed by $s_1,\ldots,s_m$ in this order, the coefficient of $x_S$ is $\sum_{\sigma\in S_m}\prod_{i=1}^m m_{i,s_{\sigma(i)}}=\operatorname{per}(M_S)$. Summing over all subsets $S\subseteq[n]$ with $|S|=m$ gives $L_1\cdots L_m=\sum_{\substack{S\subseteq[n]\\|S|=m}}\operatorname{per}(M_S)x_S$, which proves \eqref{encoding}.
\end{proof} 

For integers $1\leq k\leq n$, let $L_1,\ldots,L_{n-k}$ be independent and uniformly distributed in $(R_n)_1$, and set $\gamma_{n,k}=\Pr[L_1\cdots L_{n-k}=0]$. Since the field $\mathbb F_q$ is fixed throughout this section, we omit $q$ from the notation. When $k=n$, the product is empty and equals $1$, so $\gamma_{n,n}=0$.

The first inequality in the following lemma is Corollary~2.4 of Hunter, Kwan, and Sauermann \cite{HunterKwanSauermann2026}, rewritten in the present notation using Lemma~\ref{encodinglemma}. We include a proof in the present algebraic formulation for the reader's convenience. 

\begin{lemma}\label{taillemma}
For $1\leq k\leq n$, one has
\begin{equation}
\gamma_{n,k}
\leq
1-\prod_{j=k+1}^n(1-q^{-j})
\leq
\frac{q^{-k}}{q-1}.
\label{tailbound}
\end{equation}
\end{lemma}

\begin{proof}
When $k=n$, one has $\gamma_{n,n}=0$ and $\prod_{j=n+1}^n(1-q^{-j})=1$, so the result follows. We may therefore assume that $1\leq k\leq n-1$. Set $P_d=L_1\cdots L_d$ for $0\leq d\leq n-k$, with $P_0=1$. Since $\gamma_{n,k}=\Pr[P_{n-k}=0]$, we estimate $\Pr[P_{n-k}\neq0]$ by controlling the probability that a nonzero product becomes zero when the next random linear form is multiplied in.

Fix $0\leq d\leq n-k-1$ and let $0\neq F\in(R_n)_d$. Choose $S\subseteq[n]$ such that the coefficient $c_S$ of $x_S$ in $F$ is nonzero. Since $F$ is homogeneous of degree $d$, one has $|S|=d$. We show that the elements $x_jF$ with $j\notin S$ are linearly independent. Suppose that $\sum_{j\notin S}\alpha_jx_jF=0$ and fix $j\notin S$. The coefficient of $x_{S\cup\{j\}}$ in $x_jF$ is $c_S$. If $j'\notin S$ and $j'\neq j$, then every nonzero monomial in $x_{j'}F$ contains $x_{j'}$, while $x_{S\cup\{j\}}$ does not contain $x_{j'}$. Hence $x_{j'}F$ does not contribute to the coefficient of $x_{S\cup\{j\}}$. Comparing this coefficient in the assumed linear relation gives $\alpha_jc_S=0$. Since $c_S\neq0$, we have $\alpha_j=0$. This holds for every $j\notin S$, so the $n-d$ elements $x_jF$ with $j\notin S$ are linearly independent.

Multiplication by $F$ therefore defines a linear map from $(R_n)_1$ to $(R_n)_{d+1}$ of rank at least $n-d$. Since $(R_n)_1$ has dimension $n$, the kernel of this map has dimension at most $d$. If $L$ is uniformly distributed in $(R_n)_1$, then
\[
\Pr[LF=0]
=
\frac{\bigl|\{L\in(R_n)_1\mid LF=0\}\bigr|}{\bigl|(R_n)_1\bigr|}\leq
\frac{q^d}{q^n}
=
q^{-(n-d)}.
\]

We now apply this estimate with $F=P_d$. On the event $P_d\neq0$, the element $P_d$ is a nonzero member of $(R_n)_d$. Moreover, $L_{d+1}$ is independent of $P_d$ and remains uniformly distributed in $(R_n)_1$. It follows that
\[
\begin{aligned}
\Pr[P_{d+1}=0\mid P_d\neq0]
&=
\sum_{\substack{F\in(R_n)_d\\F\neq0}}
\Pr[P_d=F\mid P_d\neq0]\Pr[L_{d+1}F=0]\\
&\leq
q^{-(n-d)}
\sum_{\substack{F\in(R_n)_d\\F\neq0}}
\Pr[P_d=F\mid P_d\neq0]=
q^{-(n-d)}.
\end{aligned}
\]
Consequently, $\Pr[P_{d+1}\neq0\mid P_d\neq0]\geq1-q^{-(n-d)}$. Since $P_{d+1}\neq0$ implies $P_d\neq0$, we have that
\[
\begin{aligned}
1-\gamma_{n,k}
&=
\Pr[P_{n-k}\neq0]=
\prod_{d=0}^{n-k-1}
\Pr[P_{d+1}\neq0\mid P_d\neq0]\\ &\geq 
\prod_{d=0}^{n-k-1}\left(1-q^{-(n-d)}\right)=
\prod_{j=k+1}^n(1-q^{-j}).
\end{aligned}
\]
Taking complements proves the first inequality in \eqref{tailbound}. Finally, by expanding the difference between successive partial products we obtain that
\[
1-\prod_{j=k+1}^n(1-q^{-j})
=
\sum_{j=k+1}^n
q^{-j}\prod_{\ell=k+1}^{j-1}(1-q^{-\ell})\leq
\sum_{j=k+1}^n q^{-j}\leq 
\frac{q^{-k}}{q-1}.
\]
This proves the second inequality and completes the proof.
\end{proof}

We next give the character identities used in the quadratic zero count. Write $\mathbb F_q^\times=\mathbb F_q\setminus\{0\}$ and $\mathbb C^\times=\mathbb C\setminus\{0\}$. A character of a finite abelian group $G$ is a group homomorphism from $G$ to the multiplicative group $\mathbb C^\times$. It is nontrivial if it is not identically equal to $1$. An additive character of $\mathbb F_q$ is a character of the additive group $(\mathbb F_q,+)$. Fix a nontrivial additive character $\psi$, so that $\psi(u+v)=\psi(u)\psi(v)$ for all $u,v\in\mathbb F_q$. The nonzero squares form a subgroup of index two in $\mathbb F_q^\times$. Define $\chi\colon\mathbb F_q^\times\to\{-1,1\}$ by setting $\chi(u)=1$ if $u=v^2$ for some $v\in\mathbb F_q^\times$, and $\chi(u)=-1$ otherwise. Then $\chi(uv)=\chi(u)\chi(v)$ for all $u,v\in\mathbb F_q^\times$, so $\chi$ is a character of $\mathbb F_q^\times$. It is called the quadratic character of $\mathbb F_q$. We extend $\chi$ to $\mathbb F_q$ by setting $\chi(0)=0$.

\begin{lemma}\label{lem:character-identities}
For every $u\in\mathbb F_q$,
\[
\frac{1}{q}\sum_{t\in\mathbb F_q}\psi(tu)=
\begin{cases}
1,&u=0,\\
0,&u\neq0.
\end{cases}
\]
Let $\mathfrak g=\sum_{z\in\mathbb F_q}\chi(z)\psi(z)$. Then $\sum_{y\in\mathbb F_q}\psi(cy^2)=\chi(c)\mathfrak g$ for every $c\in\mathbb F_q^\times$, and $\mathfrak g^2=\chi(-1)q$. Moreover, for every positive integer $\rho$, the sum $\sum_{t\in\mathbb F_q^\times}\chi(t)^\rho$ equals $0$ when $\rho$ is odd and $q-1$ when $\rho$ is even.
\end{lemma}

These identities are standard consequences of character orthogonality and the evaluation of the quadratic Gauss sum, as developed in \cite{LidlNiederreiter1997}, so we omit the proof.

We next give the zero-count estimate for homogeneous quadratic forms that will be used below. Let $V$ be an $s$-dimensional vector space over $\mathbb F_q$. By a homogeneous quadratic form on $V$, we mean a function $f\colon V\to\mathbb F_q$ for which there is a basis $e_1,\ldots,e_s$ of $V$ such that, if the coordinates of $v\in V$ are $y_1,\ldots,y_s$, then $f(v)=\sum_{i=1}^s a_i y_i^2+\sum_{1\leq i<j\leq s}b_{ij}y_iy_j$ for some coefficients $a_i,b_{ij}\in\mathbb F_q$. These coefficients are uniquely determined by $f$ and the chosen basis, since $a_i=f(e_i)$ and $b_{ij}=f(e_i+e_j)-f(e_i)-f(e_j)$. Since $2$ is invertible in $\mathbb F_q$, there is a unique symmetric matrix $A$ satisfying $f(v)=y^{\mathsf T}Ay$ in these coordinates, with $A_{ii}=a_i$ and $A_{ij}=A_{ji}=\frac{b_{ij}}{2}$. We define the rank of $f$ to be $\operatorname{rank}(A)$. Under a change of basis, $A$ is replaced by $P^{\mathsf T}AP$ for some invertible matrix $P$, so the rank does not depend on the chosen basis.

The corresponding solution counts when the rank equals the number of variables are given in \cite{LidlNiederreiter1997}. The lemma below provides the form needed here, where the rank may be smaller than $\dim V$ and the zero count is expressed in terms of both $\dim V$ and the rank. 

\begin{lemma}\label{quadraticlemma}
Let $V$ be an $s$-dimensional vector space over $\mathbb F_q$, where $q$ is odd, and let $f\colon V\to\mathbb F_q$ be a nonzero homogeneous quadratic form of rank $\rho$. There is a sign $\eps\in\{-1,1\}$ in the even-rank case such that
\begin{equation}
\begin{aligned}
\bigl|\{v\in V\mid f(v)=0\}\bigr|
&=
\begin{cases}
q^{s-1},&\rho\text{ is odd},\\
q^{s-1}+\eps(q-1)q^{s-\rho/2-1},&\rho\text{ is even},
\end{cases}\\
\bigl|\{v\in V\mid f(v)=0\}\bigr|
&\leq(2q-1)q^{s-2}.
\end{aligned}
\label{quadraticcount}
\end{equation}
\end{lemma}

\begin{proof}
Since $q$ is odd, a change of basis diagonalizes the symmetric matrix representing $f$. We may therefore write $f(y)=a_1y_1^2+\cdots+a_\rho y_\rho^2$, where $a_1,\ldots,a_\rho\in\mathbb F_q^\times$. Let $\psi$, $\chi$, and $\mathfrak g$ be as in Lemma~\ref{lem:character-identities}. Character orthogonality gives the first equality below. The term with $t=0$ contributes $q^{s-1}$. For $t\in\mathbb F_q^\times$, the $s-\rho$ coordinates absent from $f$ contribute $q^{s-\rho}$, and the remaining sums factor. Therefore, by Lemma~\ref{lem:character-identities} we obtain that
\[
\begin{aligned}
\bigl|\{v\in V\mid f(v)=0\}\bigr|
&=\frac{1}{q}\sum_{t\in\mathbb F_q}\sum_{y\in\mathbb F_q^s}\psi(tf(y))\\
&=q^{s-1}+q^{s-\rho-1}\sum_{t\in\mathbb F_q^\times}\prod_{i=1}^{\rho}\left(\sum_{u\in\mathbb F_q}\psi(ta_i u^2)\right)\\
&=q^{s-1}+q^{s-\rho-1}\mathfrak g^\rho\left(\prod_{i=1}^{\rho}\chi(a_i)\right)\sum_{t\in\mathbb F_q^\times}\chi(t)^\rho.
\end{aligned}
\]
If $\rho$ is odd, the final sum is zero, so the number of zeros is $q^{s-1}$. If $\rho$ is even, the final sum is $q-1$, while $\mathfrak g^\rho=(\chi(-1)q)^{\rho/2}$. Hence the number of zeros is $q^{s-1}+\varepsilon(q-1)q^{s-\rho/2-1}$, where $\varepsilon=\chi\left((-1)^{\rho/2}a_1\cdots a_\rho\right)\in\{-1,1\}$.

It remains to establish the upper bound. If $\rho$ is odd, then $q^{s-1}\leq(2q-1)q^{s-2}$. If $\rho$ is even, then $f\neq0$ implies $\rho\geq2$, and consequently
\[
\bigl|\{v\in V\mid f(v)=0\}\bigr|\leq q^{s-1}+(q-1)q^{s-2}=(2q-1)q^{s-2}.
\]
\end{proof}

\section{Coordinate collisions and dimension reduction}\label{collisionsection}
For $a\in R_n$, define the annihilator of $a$ by $\operatorname{Ann}_{R_n}(a)=\{r\in R_n:ar=0\}$, and define the principal ideal generated by $a$ by $aR_n=\{ar:r\in R_n\}$. We first determine the annihilator of $x_i-x_j$ and the quotient obtained by imposing the relation $x_i+x_j=0$. In this quotient, the relation $x_j=-x_i$ allows one generator to be eliminated. The following lemma shows that the resulting quotient is a square-zero algebra on $n-1$ generators and thereby makes the reduction from $R_n$ to $R_{n-1}$ precise.

\begin{lemma}\label{annihilatorlemma}
For distinct $i,j\in[n]$, the following identities hold.
\begin{equation}
\operatorname{Ann}_{R_n}(x_i-x_j)=(x_i+x_j)R_n
\qquad\text{and}\qquad
R_n/(x_i+x_j)R_n\cong R_{n-1}.
\label{annihilatoridentity}
\end{equation}
Moreover, the isomorphism in~(\ref{annihilatoridentity}) may be chosen so that, for every $d\geq0$, it maps the residue classes represented by elements of $(R_n)_d$ onto $(R_{n-1})_d$.
\end{lemma}

\begin{proof}
By definition, $\operatorname{Ann}_{R_n}(x_i-x_j)=\{F\in R_n:(x_i-x_j)F=0\}$. We therefore determine all $F\in R_n$ satisfying $(x_i-x_j)F=0$.
Let $\mathcal S$ be the $\mathbb F_q$-subspace spanned by the basis monomials $x_T$ with $T\subseteq[n]\setminus\{i,j\}$. It is also the subalgebra generated by the variables $x_k$ with $k\notin\{i,j\}$. Partitioning the basis monomials of $R_n$ according to whether they contain $x_i$ and $x_j$ gives $R_n=\mathcal S\oplus x_i\mathcal S\oplus x_j\mathcal S\oplus x_ix_j\mathcal S$. Hence every $F\in R_n$ has a unique expression $F=A+x_iB+x_jC+x_ix_jD$, where $A,B,C,D\in\mathcal S$.
Using $x_i^2=x_j^2=0$, we obtain $(x_i-x_j)F=x_iA-x_jA+x_ix_j(C-B)$. The three terms on the right belong respectively to the distinct summands $x_i\mathcal S$, $x_j\mathcal S$, and $x_ix_j\mathcal S$. Moreover, multiplication by $x_i$, $x_j$, or $x_ix_j$ maps distinct basis monomials of $\mathcal S$ to distinct basis monomials of $R_n$. It follows that $(x_i-x_j)F=0$ if and only if $A=0$ and $C=B$. In that case, $F=x_iB+x_jB+x_ix_jD=(x_i+x_j)(B+x_iD)$, so $F\in(x_i+x_j)R_n$. Thus $\operatorname{Ann}_{R_n}(x_i-x_j)\subseteq(x_i+x_j)R_n$.
Conversely, if $F\in(x_i+x_j)R_n$, then $F=(x_i+x_j)G$ for some $G\in R_n$. Hence $(x_i-x_j)F=(x_i-x_j)(x_i+x_j)G=(x_i^2-x_j^2)G=0$, so $F\in\operatorname{Ann}_{R_n}(x_i-x_j)$. This proves the first identity in \eqref{annihilatoridentity}.

For the second identity, set $I=(x_i+x_j)R_n$, write $\overline{x}_k=x_k+I\in R_n/I$ for each $k\in [n]$, and let $T=\mathbb F_q[y_k:k\in[n]\setminus\{j\}]/(y_k^2:k\in[n]\setminus\{j\})$. After relabeling its generators, $T$ is $R_{n-1}$. We first construct a surjective homomorphism from $T$ to $R_n/I$ and then prove that it is injective by constructing a left inverse.
Since $\overline{x}_k^{\,2}=0$ for every $k\neq j$, the assignments $y_k\mapsto\overline{x}_k$ define an $\mathbb F_q$-algebra homomorphism $\Phi\colon T\to R_n/I$. Moreover, $\overline{x}_j=-\overline{x}_i$, so all the generators $\overline{x}_1,\ldots,\overline{x}_n$ of $R_n/I$ lie in the image of $\Phi$. Hence $\Phi$ is surjective.
To prove that $\Phi$ is injective, define $\Psi\colon R_n\to T$ by $\Psi(x_k)=y_k$ for $k\neq j$ and $\Psi(x_j)=-y_i$. This homomorphism is well defined because $y_k^2=0$ for every $k\neq j$ and $(-y_i)^2=0$. Since $\Psi(x_i+x_j)=y_i-y_i=0$, we have $I\subseteq\ker\Psi$, and hence $\Psi$ induces an $\mathbb F_q$-algebra homomorphism $\overline{\Psi}\colon R_n/I\to T$. For every $k\neq j$, $(\overline{\Psi}\circ\Phi)(y_k)=y_k$, and therefore $\overline{\Psi}\circ\Phi$ is the identity map on $T$. If $\Phi(F)=0$ for some $F\in T$, applying $\overline{\Psi}$ gives $F=(\overline{\Psi}\circ\Phi)(F)=0$. Thus $\Phi$ is injective and hence is an isomorphism. Consequently, $R_n/I\cong T\cong R_{n-1}$.
Finally, $\overline{\Psi}$ replaces each occurrence of $x_j$ by $-y_i$ and each occurrence of $x_k$ with $k\neq j$ by $y_k$. It therefore maps, for every $d\geq0$, the residue classes represented by elements of $(R_n)_d$ onto the degree-$d$ part of $T$, which is identified with $(R_{n-1})_d$.
\end{proof}

The next lemma is the core recurrence. The two indices have different roles. The first records the number of variables, while the second records the number of omitted linear factors.

\begin{lemma}\label{recurrenceproposition}
For $n\geq2$ and $1\leq k\leq n-1$, one has
\begin{equation}
\gamma_{n,k}
\leq
\frac{2q-1}{q^2}\left(\frac{1}{n}+\frac{n-1}{n}\gamma_{n-1,k}\right)
+
\frac{(q-1)^2}{q^2}\gamma_{n,k+1}.
\label{recurrence}
\end{equation}
\end{lemma}

\begin{proof}
Let $G=L_1\cdots L_{n-k-1}$, where the product is understood to be $1$ when $k=n-1$. For each realization of $G$, define the linear map $T_G\colon(R_n)_1\to(R_n)_{n-k}$ by $T_G(u)=uG$, and write $s(G)=\operatorname{rank}T_G$. Since $(R_n)_1$ has dimension $n$, by rank--nullity we have $|\ker T_G|=q^{n-s(G)}$. Conditional on $G$, the final form $L_{n-k}$ is still uniform on $(R_n)_1$ and is independent of $G$. Consequently, for each possible value $g$ of $G$,
$
\Pr[L_{n-k}G=0\mid G=g]=\frac{|\ker T_g|}{q^n}=q^{-s(g)}.
$
Averaging this conditional probability over the distribution of $G$ and applying the law of total probability, we obtain
\[
\gamma_{n,k}
=
\Pr[L_{n-k}G=0]
=
\mathbb E\!\left[\Pr[L_{n-k}G=0\mid G]\right]
=
\mathbb E[q^{-s(G)}].
\] Set $U=(R_n)_1/\ker T_G$, let $v_i=x_i+\ker T_G\in U$, and define the ordered collision density by $D(G)=n^{-2}|\{(i,j)\in[n]^2:v_i=v_j\}|$. Thus $\dim U=s(G)$. Moreover, since $v_i=x_i+\ker T_G$, we have
\[
D(G)
=
\frac{1}{n^2}
\left|\{(i,j)\in[n]^2:v_i=v_j\}\right|
=
\frac{1}{n^2}
\left|\{(i,j)\in[n]^2:(x_i-x_j)G=0\}\right|.
\]

Suppose that $G\neq0$. Since $G$ is homogeneous of degree $n-k-1$, there is some $S\subseteq[n]$ with $|S|=n-k-1$ such that the coefficient $c_S$ of $x_S$ in $G$ is nonzero. The elements $x_jG$ with $j\notin S$ are linearly independent. Indeed, suppose that $\sum_{j\notin S}a_jx_jG=0$ and fix $r\notin S$. The coefficient of $x_{S\cup\{r\}}$ in $x_rG$ is $c_S$. If $j\notin S$ and $j\neq r$, then every nonzero monomial in $x_jG$ contains $x_j$, whereas $x_{S\cup\{r\}}$ does not contain $x_j$; hence $x_jG$ does not contribute to this coefficient. Comparing coefficients gives $a_rc_S=0$, so $a_r=0$. This holds for every $r\notin S$, proving the asserted linear independence. Since there are $|[n]\setminus S|=k+1$ such elements in the image of $T_G$, we obtain $s(G)\geq k+1\geq2$.

We next show that the elements $v_1,\ldots,v_n$ satisfy a common nonzero quadratic equation on $U$; the zero-count estimate in Lemma~\ref{quadraticlemma} will then bound the number of distinct values among these elements and thereby give a lower bound for $D(G)$.

\begin{claim}\label{quadraticclaim}
If $G\neq0$, there exists a nonzero homogeneous quadratic form $f\colon U\to\mathbb F_q$ such that $f(v_i)=0$ for every $i\in[n]$.
\end{claim}

\begin{proof}[Proof of Claim~\ref{quadraticclaim}]
To prove the claim, keep the set $S$ chosen above and choose distinct indices $a,b\notin S$, which is possible because $|[n]\setminus S|=k+1\geq2$. Every element $H\in(R_n)_{n-k+1}$ has a unique expansion
$
H=\sum_{\substack{A\subseteq[n]\\|A|=n-k+1}}c_Ax_A.
$
Define $\lambda\colon(R_n)_{n-k+1}\to\mathbb F_q$ by
$
\lambda(H)=c_{S\cup\{a,b\}};
$
thus $\lambda(H)$ is the coefficient of $x_{S\cup\{a,b\}}$ in $H$. Since $u^2G\in(R_n)_{n-k+1}$ for every $u\in(R_n)_1$, consider the formula
$
f(u+\ker T_G)=\lambda(u^2G).
$
We verify that this formula depends only on the coset $u+\ker T_G$. If $u+w$ is another representative of this coset, then $w\in\ker T_G$, so $wG=0$. Consequently,
\[
(u+w)^2G-u^2G
=
2uwG+w^2G
=
2u(wG)+w(wG)
=
0.
\]
It follows that $\lambda((u+w)^2G)=\lambda(u^2G)$, and hence the formula defines a function $f\colon U\to\mathbb F_q$. To see directly that $f$ is a homogeneous quadratic form, choose a basis $e_1,\ldots,e_{s(G)}$ of $U$ and representatives $\widetilde e_1,\ldots,\widetilde e_{s(G)}$ in $(R_n)_1$. Then
\[
f\left(\sum_{\nu=1}^{s(G)}t_\nu e_\nu\right)
=
\sum_{\mu,\nu=1}^{s(G)}
t_\mu t_\nu\lambda(\widetilde e_\mu\widetilde e_\nu G),
\]
which is a homogeneous polynomial of degree two in the coordinates $t_1,\ldots,t_{s(G)}$. Moreover, $f(v_\ell)=\lambda(x_\ell^2G)=0$ for every $\ell\in[n]$. On the other hand, the coefficient of $x_{S\cup\{a,b\}}$ in $x_ax_bG$ is exactly $c_S$, because no monomial of $G$ other than $x_S$ can produce $x_{S\cup\{a,b\}}$ after multiplication by $x_ax_b$. Since $q$ is odd, we have that 
\[
f(v_a+v_b)
=
\lambda((x_a+x_b)^2G)
=
2\lambda(x_ax_bG)
=
2c_S
\neq0.
\]
Thus $f$ is nonzero, proving the claim.
\end{proof}

Let the distinct values among $v_1,\ldots,v_n$ be $w_1,\ldots,w_h$, and let $m_\nu=|\{i\in[n]:v_i=w_\nu\}|$. Then $\sum_{\nu=1}^hm_\nu=n$. For each $\nu$, there are exactly $m_\nu^2$ ordered pairs $(i,j)$ such that $v_i=v_j=w_\nu$. Hence we obtain that
$
D(G)=\frac1{n^2}\sum_{\nu=1}^hm_\nu^2.
$
By the Cauchy--Schwarz inequality,
$
n^2=\left(\sum_{\nu=1}^hm_\nu\right)^2\leq h\sum_{\nu=1}^hm_\nu^2,
$
and therefore $D(G)\geq1/h$. Moreover, each $w_\nu$ satisfies $f(w_\nu)=0$. Since $f$ is a nonzero quadratic form on the $s(G)$-dimensional space $U$, by Lemma~\ref{quadraticlemma} we have
$
h\leq(2q-1)q^{s(G)-2}.
$
Consequently,
\[
D(G)\geq\frac1h\geq\frac1{(2q-1)q^{s(G)-2}}=\frac{q^2}{2q-1}q^{-s(G)},
\]
which implies that $q^{-s(G)}\leq(2q-1)D(G)/q^2$ whenever $G\neq0$. If $G=0$, then $T_G=0$, so $s(G)=0$, $U=\{0\}$, and all the classes $v_i$ coincide; hence $D(G)=1$. Since $1-(2q-1)/q^2=(q-1)^2/q^2$, the two cases combine into the pointwise inequality
\begin{equation}
q^{-s(G)}
\leq
\frac{2q-1}{q^2}D(G)
+
\frac{(q-1)^2}{q^2}\1_{\{G=0\}}.
\label{pointwisebound}
\end{equation}
Here $(2q-1)/q^2$ comes from the quadratic zero-count estimate, whereas $(q-1)^2/q^2$ is the complementary amount needed to make the inequality valid in the exceptional case $G=0$.

It remains to compute $\mathbb E[D(G)]$. By definition, $\mathbb E[D(G)]=n^{-2}\sum_{i,j=1}^n\Pr[v_i=v_j]$. The equality $v_i=v_i$ holds identically for each of the $n$ diagonal pairs. Now fix distinct $i,j\in[n]$. By the definition of $U$ and Lemma~\ref{annihilatorlemma}, one has $v_i=v_j$ if and only if $(x_i-x_j)G=0$, which is equivalent to $G\in(x_i+x_j)R_n$, and hence to the image of $G$ being zero under the quotient map $\pi_{ij}\colon R_n\to R_n/(x_i+x_j)R_n\cong R_{n-1}$. Write the generators of this copy of $R_{n-1}$ as $y_\ell$, with $\ell\in[n]\setminus\{j\}$. If $L_r=\sum_{\ell=1}^na_{r\ell}x_\ell$, then
\[
\pi_{ij}(L_r)
=
(a_{ri}-a_{rj})y_i
+
\sum_{\ell\notin\{i,j\}}a_{r\ell}y_\ell.
\]
The corresponding coefficient map from $\mathbb F_q^n$ to $\mathbb F_q^{n-1}$ is surjective, and every target vector has exactly $q$ preimages: the coefficient $a_{rj}$ may be chosen arbitrarily, after which $a_{ri}$ and all the remaining coefficients are determined. It therefore sends a uniform linear form in $(R_n)_1$ to a uniform linear form in $(R_{n-1})_1$. Applying the same deterministic map separately to the independent forms $L_1,\ldots,L_{n-k-1}$ also preserves their independence. Since $n-k-1=(n-1)-k$, their images are precisely the number of independent uniform linear forms used in the definition of $\gamma_{n-1,k}$. Thus $\Pr[v_i=v_j]=\Pr[\pi_{ij}(G)=0]=\gamma_{n-1,k}$ for every $i\neq j$. There are $n$ diagonal ordered pairs and $n(n-1)$ off-diagonal ordered pairs, so
\[
\mathbb E[D(G)]
=
\frac{n+n(n-1)\gamma_{n-1,k}}{n^2}
=
\frac1n+\frac{n-1}{n}\gamma_{n-1,k}.
\]

Finally, $G$ is the product of $n-k-1=n-(k+1)$ independent uniform linear forms in $R_n$, so $\Pr[G=0]=\gamma_{n,k+1}$. Taking expectations in \eqref{pointwisebound} and using the exact identity $\gamma_{n,k}=\mathbb E[q^{-s(G)}]$, we obtain that
\[
\begin{aligned}
\gamma_{n,k}
&\leq
\frac{2q-1}{q^2}\mathbb E[D(G)]
+
\frac{(q-1)^2}{q^2}\Pr[G=0]\\
&=
\frac{2q-1}{q^2}
\left(
\frac1n+\frac{n-1}{n}\gamma_{n-1,k}
\right)
+
\frac{(q-1)^2}{q^2}\gamma_{n,k+1},
\end{aligned}
\]
which is \eqref{recurrence}. The argument also includes the boundary case $k=n-1$, in which case $G=1$, while $\gamma_{n,n}=\gamma_{n-1,n-1}=0$ by the empty-product convention.
\end{proof}

\section{Iteration and completion of the proof}\label{iterationsection}
We first iterate Lemma~\ref{recurrenceproposition}. The proof keeps track of the full range of admissible indices, which will be needed when the iteration reaches the diagonal $k=n$.

\begin{proposition}[Iterated recurrence]\label{iterationproposition}
For $1\leq k\leq n$ and every integer $t$ with $0\leq t\leq n-k$, one has
\begin{equation}
\gamma_{n,k}
\leq
\frac{(2q-1)t}{q^2n}
+
\frac{q^{-k}}{q-1}
\left(1-\left(1-\frac{1}{q}\right)^3\right)^t.
\label{iteratedbound}
\end{equation}
\end{proposition}

\begin{proof}
Set $a=(2q-1)/q^2$ and $b=(q-1)^2/q^2$, so that $a+b=1$. We first iterate \eqref{recurrence} $r$ times, separating the accumulated contribution of its constant term from the terms remaining after the $r$ iterations. We will then take $r=t$ and apply Lemma~\ref{taillemma} to the remaining terms.

\begin{claim}\label{rstepclaim}
For every integer $r$ with $0\leq r\leq n-k$,
\[
\gamma_{n,k}
\leq
\frac{ar}{n}
+
\sum_{j=0}^r
\binom{r}{j}
a^jb^{r-j}
\frac{n-j}{n}
\gamma_{n-j,k+r-j}.
\]
\end{claim}

Here $j$ records the number of iterations in which the first index is decreased; in the remaining $r-j$ iterations, the second index is increased.

\begin{proof}[Proof of Claim~\ref{rstepclaim}]
We proceed by induction on $r$. When $r=0$, the asserted inequality is an equality. Suppose that it holds for some $r<n-k$. For every $0\leq j\leq r$, we have $(n-j)-(k+r-j)=n-k-r\geq1$. Hence \eqref{recurrence}, with $n$ replaced by $n-j$ and $k$ replaced by $k+r-j$, gives
\[
\gamma_{n-j,k+r-j}
\leq
\frac{a}{n-j}
+
a\frac{n-j-1}{n-j}\gamma_{n-j-1,k+r-j}
+
b\gamma_{n-j,k+r-j+1}.
\]
By the induction hypothesis,
\[
\gamma_{n,k}
\leq
\frac{ar}{n}
+
\sum_{j=0}^r
\binom{r}{j}
a^jb^{r-j}
\frac{n-j}{n}
\gamma_{n-j,k+r-j}.
\]
All the coefficients in this sum are nonnegative. Substituting the preceding bound for $\gamma_{n-j,k+r-j}$ into the right-hand side therefore yields
\[
\begin{aligned}
\gamma_{n,k}
&\leq \frac{ar}{n} + \sum_{j=0}^r \binom{r}{j} a^jb^{r-j} \frac{n-j}{n} \left( \frac{a}{n-j} + a\frac{n-j-1}{n-j}\gamma_{n-j-1,k+r-j}+ b\gamma_{n-j,k+r-j+1} \right)\\
&= \frac{ar}{n}+\frac{a}{n}\sum_{j=0}^r\binom{r}{j}a^jb^{r-j}+\sum_{j=0}^r
\binom{r}{j}
a^{j+1}b^{r-j}
\frac{n-j-1}{n}
\gamma_{n-j-1,k+r-j}\\& \qquad\qquad\qquad\qquad\quad\qquad +
\sum_{j=0}^r
\binom{r}{j}
a^jb^{r+1-j}
\frac{n-j}{n}
\gamma_{n-j,k+r-j+1}.
\end{aligned}
\]
Because $a+b=1$, the first sum in the last expression satisfies
$\frac{a}{n}\sum_{j=0}^r\binom{r}{j}a^jb^{r-j}=\frac{a}{n}(a+b)^r=\frac{a}{n}.$
In the sum containing $\gamma_{n-j-1,k+r-j}$, set $\ell=j+1$, and in the last sum rename $j$ as $\ell$. The preceding inequality then becomes
\[
\begin{aligned}
\gamma_{n,k}
&\leq
\frac{a(r+1)}{n}+
\sum_{\ell=1}^{r+1}
\binom{r}{\ell-1}
a^\ell b^{r+1-\ell}
\frac{n-\ell}{n}
\gamma_{n-\ell,k+r+1-\ell}+
\sum_{\ell=0}^{r}
\binom{r}{\ell}
a^\ell b^{r+1-\ell}
\frac{n-\ell}{n}
\gamma_{n-\ell,k+r+1-\ell}\\
&=
\frac{a(r+1)}{n}
+
\sum_{\ell=0}^{r+1}
\binom{r+1}{\ell}
a^\ell b^{r+1-\ell}
\frac{n-\ell}{n}
\gamma_{n-\ell,k+r+1-\ell}.
\end{aligned}
\]
Indeed, for $1\leq\ell\leq r$, the two sums are combined using
$
\binom{r}{\ell-1}+\binom{r}{\ell}=\binom{r+1}{\ell},
$
while the terms with $\ell=0$ and $\ell=r+1$ occur only in the second and first sums, respectively. After renaming $\ell$ as $j$, the resulting inequality is precisely the asserted inequality with $r+1$ in place of $r$. This completes the induction.
\end{proof}

Apply Claim~\ref{rstepclaim} with $r=t$. For every $0\leq j\leq t$, we have $(n-j)-(k+t-j)=n-k-t\geq0$, so Lemma~\ref{taillemma} gives $\gamma_{n-j,k+t-j}\leq q^{-(k+t-j)}/(q-1)$. Since $(n-j)/n\leq1$, it follows that
\[
\begin{aligned}
\gamma_{n,k}
&\leq
\frac{at}{n}
+
\frac{1}{q-1}
\sum_{j=0}^t
\binom{t}{j}
a^jb^{t-j}
\frac{n-j}{n}
q^{-(k+t-j)}\\
&\leq
\frac{at}{n}
+
\frac{q^{-k}}{q-1}
\sum_{j=0}^t
\binom{t}{j}
a^j\left(\frac{b}{q}\right)^{t-j}=
\frac{at}{n}
+
\frac{q^{-k}}{q-1}
\left(a+\frac{b}{q}\right)^t.
\end{aligned}
\]
Finally, $a+b/q=1-b(1-1/q)=1-(1-1/q)^3$. Substituting this identity and $a=(2q-1)/q^2$ into the preceding inequality gives \eqref{iteratedbound}.
\end{proof}

We are now ready to prove Theorem~\ref{maintheorem}.

\begin{proof}
The case $n=1$ is immediate, since $\operatorname{per}(A_1)$ is uniform on $\mathbb F_q$. Assume that $n\geq2$. Write the last row of $A_n$ as $y=(y_1,\ldots,y_n)$ and let $B$ be the matrix formed by the first $n-1$ rows. For $j\in[n]$, let $B^{(j)}$ be the matrix obtained from $B$ by deleting its $j$th column, and set $c_j=\operatorname{per}(B^{(j)})$ and $c=(c_1,\ldots,c_n)$. By the definition of the permanent, expansion along the last row gives $\operatorname{per}(A_n)=\sum_{j=1}^n c_jy_j$.
For $1\leq i\leq n-1$, write the $i$th row of $B$ as $(b_{i1},\ldots,b_{in})$ and set $L_i=\sum_{j=1}^n b_{ij}x_j$. Applying Lemma~\ref{encodinglemma} with $m=n-1$ and $M=B$, we obtain that
\[
L_1\cdots L_{n-1}
=
\sum_{j=1}^n\operatorname{per}(B^{(j)})x_{[n]\setminus\{j\}}
=
\sum_{j=1}^n c_jx_{[n]\setminus\{j\}}.
\]
The elements $x_{[n]\setminus\{j\}}$, $j\in[n]$, are linearly independent, and hence $L_1\cdots L_{n-1}=0$ if and only if $c=0$. Since the rows of $B$ are independent and uniform on $\mathbb F_q^n$, the elements $L_1,\ldots,L_{n-1}$ are independent and uniform on $(R_n)_1$. Therefore $\Pr(c=0)=\Pr(L_1\cdots L_{n-1}=0)=\gamma_{n,1}$.

Conditional on $B$, the vector $c$ is fixed, while $y$ remains uniform on $\mathbb F_q^n$. If $c=0$, then $\operatorname{per}(A_n)=0$. If $c\neq0$, then for every $z\in\mathbb F_q$, the equation $\sum_{j=1}^n c_jy_j=z$ has exactly $q^{n-1}$ solutions in $\mathbb F_q^n$, so $\Pr(\operatorname{per}(A_n)=z\mid B)=q^{-1}$. Averaging over $B$ yields
\[
\Pr(\operatorname{per}(A_n)=z)
=
\begin{cases}
\displaystyle \frac1q+\left(1-\frac1q\right)\gamma_{n,1},&z=0,\\[6pt]
\displaystyle \frac1q-\frac{\gamma_{n,1}}q,&z\neq0.
\end{cases}
\]
Consequently,
\[
d_{\mathrm{TV}}\!\left(\mathcal L(\operatorname{per}(A_n)),\mathsf U_q\right)
=
\frac12\left(
\left(1-\frac1q\right)\gamma_{n,1}
+
(q-1)\frac{\gamma_{n,1}}q
\right)=
\left(1-\frac1q\right)\gamma_{n,1}.
\]

Applying Proposition~\ref{iterationproposition} with $k=1$, we have that for every $0\leq t\leq n-1$,
\[
\gamma_{n,1}
\leq
\frac{(2q-1)t}{q^2n}
+
\frac{q^{-1}}{q-1}
\left(
1-\left(1-\frac1q\right)^3
\right)^t.
\]
It follows that
\[
\begin{aligned}
d_{\mathrm{TV}}\!\left(\mathcal L(\operatorname{per}(A_n)),\mathsf U_q\right)
&\leq
\left(1-\frac1q\right)
\left[
\frac{(2q-1)t}{q^2n}
+
\frac{q^{-1}}{q-1}
\left(
1-\left(1-\frac1q\right)^3
\right)^t
\right]\\
&=
\frac{(q-1)(2q-1)}{q^3}\frac{t}{n}
+
\frac1{q^2}
\left(
1-\left(1-\frac1q\right)^3
\right)^t,
\end{aligned}
\]
as required.
\end{proof}

The Corollary~\ref{uniformitycorollary} follows from \eqref{parameterbound} by choosing the parameter $t$ independently of $q$.

\begin{proof}[Proof of Corollary~\ref{uniformitycorollary}]
For every $q\geq3$, we have
$\frac{(q-1)(2q-1)}{q^3}\leq\frac{10}{27}$, and $ 1-\left(1-\frac1q\right)^3\leq\frac{19}{27}$.
For $n\geq7$, choose $t=\ceil{\log n/\log(27/19)}$. Then $t\leq n-1$, $(\frac{19}{27})^t\leq n^{-1}$, and $t\leq\log n/\log(27/19)+1$. Substituting these estimates into \eqref{parameterbound} obtains the desired result \eqref{uniformbound}.
The right-hand side of \eqref{uniformbound} tends to zero as $n\to\infty$, uniformly over all odd prime powers $q$. Moreover, for every $x\in\mathbb F_q$,
\[
\left|\Pr[\operatorname{per}(A_n)=x]-\frac1q\right|
\leq
d_{\mathrm{TV}}\!\left(\mathcal L(\operatorname{per}(A_n)),\mathsf U_q\right).
\]
This proves \eqref{uniformitystatement} and hence Conjecture~\ref{uniformityconjecture}; the uniformity of the bound also gives the same conclusion for every sequence $q=q(n)$ of odd prime powers.
\end{proof}

\section*{Declaration on the Use of AI}

During the development of the proof in
Section~\ref{collisionsection}, the authors used generative AI tools as research aids. The conditional-probability formulation leading to the map $T_G$, the recognition that an approach using only one additional row would result in a circular argument, and the resulting introduction of a second row and the associated bilinear form were developed by the authors. At that stage, a generative AI tool suggested diagonalizing the associated quadratic form and applying finite-field point-count estimates to control coordinate collisions. This suggestion helped lead to the dimension-reduction recurrence used in the proof. The authors subsequently worked out and independently verified all details of the argument and take full responsibility for the correctness and content of this article.

\bibliographystyle{plain}
\bibliography{refs}
\end{document}